\documentclass[11pt]{amsart}
\usepackage{bm}
\usepackage{amssymb}
\usepackage{graphicx}	
\usepackage{xcolor}
\usepackage{amsmath}
\usepackage{enumerate}
\usepackage{blindtext}
\usepackage{multirow}
\usepackage{mathrsfs}

\newtheorem{thm}{Theorem}[section]
\newtheorem{Def}[thm]{Definition}
\newtheorem{Rem}[thm]{Remark}
\newtheorem{Cor}[thm]{Corollary}

\newtheorem{Pro}[thm]{Proposition}

\newtheorem{prob}[thm]{Problem}
\newtheorem{lemma}[thm]{Lemma}
\newcommand{\bdfn}{\begin{Def} \rm}
\newcommand{\edfn}{\end{Def}}

\newcommand{\tfae}{the following are equivalent}

\newcommand{\ra}{\rightarrow}
\newcommand{\Ra}{\Rightarrow}

\newcommand{\es}{\emptyset}
\newcommand{\ci}{\subseteq}

\newcommand{\al}{\alpha}

\newcommand{\de}{\delta}
\newcommand{\e}{\varepsilon}
\newcommand{\vp}{\varphi}

\newcommand{\la}{\lambda}
\newcommand{\si}{\sigma}
\newcommand{\Si}{\Sigma}
\newcommand{\ga}{\gamma}
\newcommand{\Ga}{\Gamma}

\newcommand{\mb}{\mathbb}
\newcommand{\mc}{\mathcal}
\newcommand{\mf}{\mathfrak}
\newcommand{\mr}{\mathscr}
\newcommand{\sm}{\setminus}

\newcommand{\iy}{\infty}

\newcommand{\Om}{\Omega}
\newcommand{\tr}{\textrm}
\newcommand{\ov}{\overline}
\newcommand{\beqa}{\begin{eqnarray*}}
\newcommand{\eeqa}{\end{eqnarray*}}
\newcommand{\vertiii}[1]{{\left\vert\kern-0.25ex\left\vert\kern-0.25ex\left\vert #1 
    \right\vert\kern-0.25ex\right\vert\kern-0.25ex\right\vert}}
\newcounter{cnt1}
\newcounter{cnt2}
\newcounter{cnt3}
\newcounter{cnt4}
\newcommand{\blr}{\begin{list}{$($\roman{cnt1}$)$} {\usecounter{cnt1}
\setlength{\topsep}{0pt} \setlength{\itemsep}{0pt}}}
\newcommand{\blR}{\begin{list}{\Roman{cnt4}.\ } {\usecounter{cnt4}
\setlength{\topsep}{0pt} \setlength{\itemsep}{0pt}}}
\newcommand{\bla}{\begin{list}{$(\alph{cnt2})$} {\usecounter{cnt2}
\setlength{\topsep}{0pt} \setlength{\itemsep}{0pt}}}
\newcommand{\bln}{\begin{list}{$($\arabic{cnt3}$)$} {\usecounter{cnt3}
\setlength{\topsep}{0pt} \setlength{\itemsep}{0pt}}}
\newcommand{\el}{\end{list}}

\begin{document}
\title[Generalized centers in Banach spaces]{A study on various generalizations of Generalized centers $\bm{(GC)}$ in Banach spaces}
\author[Das]{Syamantak Das}
\author[Paul]{Tanmoy Paul}
\address{Dept. of Mathematics\\
Indian Institute of Technology Hyderabad\\
India}
\email{ma20resch11006@iith.ac.in \& tanmoy@math.iith.ac.in}
\subjclass[2000]{Primary 41A28, 41A65 Secondary 46B20 41A50 \hfill \textbf{\today} }

\begin{abstract}
In [{\em Generalized centers of finite sets in Banach spaces}, Acta Math. Univ. Comenian. (N.S.) {\bf 66}(1) (1997), 83--115], Vesel\'{y} developed the idea of generalized centers for finite sets in Banach spaces. In this work, we explore the concept of {\it restricted $\mr{F}$-center property} for a triplet $(X,Y,\mc{F}(X))$, where $Y$ is a subspace of a Banach space $X$ and $\mc{F}(X)$ is the family of finite subsets of $X$. In addition, we generalize the analysis to include all closed, bounded subsets of $X$. Similar to how Lindenstrauss characterized $n.2.I.P.$, we characterize $n.X.I.P.$. So, it is possible to figure out that $Y$ has $n.X.I.P.$ in $X$ for all natural numbers $n$ if and only if $\tr{rad}_Y(F)=\tr{rad}_X(F)$ for all finite subsets $F$ of $Y$. It then turns out that, for all continuous, monotone functions $f$, the $f$-radii viz. $\tr{rad}_Y^f(F),\tr{rad}_X^f(F)$ are same whenever the generalized radii viz. $\tr{rad}_Y(F), \tr{rad}_X(F)$ are also same, for all finite subsets $F$ of $Y$. We establish a variety of characterizations of central subspaces of Banach spaces. With reference to an appropriate subfamily of closed and bounded subsets, it appears that a number of function spaces and subspaces exhibit the restricted weighted Chebyshev center property. 
\end{abstract}

\maketitle

\section{Introduction}
\subsection{Objectives:~}
In this study, we focus on the minimization problem $\inf_{y\in Y}f(\|z_1-y\|,\ldots,\|z_n-y\|)$ within the framework of Banach spaces $X$ and its closed subspaces $Y$. The problem involves finding the minimum value of a given function $f:\mb{R}^n_{\geq 0}\ra\mb{R}_{\geq 0}$, which is continuous, monotone (point wise), coercive,   corresponding to a finite subset $\{z_1,z_2,\ldots,z_n\}$ of $X$. Additionally, we explore variations of this problem in Banach spaces. The outcomes in \cite{LV} make it clear that this phenomenon is connected to the intersection properties of balls in Banach spaces. We extend the concept of $f$-centers, as introduced in the work by Vesel\'{y} in \cite{LV}, to encompass all closed and bounded subsets of the space $X$.  

In \cite{LV}, the author observed that the generalized radius for a finite set in a Banach space remains same when it is considered in the bidual of the space. In this paper, we investigate the following problem.

\begin{prob}\label{Q1}
Let $Y$ be a subspace of $X$. For every finite subset $F$ of $Y$, under what necessary and sufficient conditions do the generalized radii of $F$ remain the same when it is viewed as a subset of both $X$ and $Y$? 
\end{prob}

Additionally, we aim to study various properties viz. {\bf $\mr{F}$-rcp}, {\bf wrcp}, $\mc{A}-C$-subspace, $(GC)$, $\mc{A}$-IP (see Definition~\ref{D1}, \ref{D2}) in various function spaces and their subspaces.

\subsection{Prerequisites:~}
We list some standard notations used in this study: $X$ denotes a Banach space, and by subspace, we indicate a closed linear subspace.  $B_X$ and $S_X$ represent the closed unit ball and the unit sphere of $X$, respectively. For $x\in X$ and $r>0$, $B_X(x,r)$ denotes the closed ball in $X$, centered at $x$ and radius $r$. When there is no chance of confusion, we simply write $B(x,r)$ for the closed ball in $X$. We consider an element $x\in X$ to be canonically embedded in $X^{**}$. $\mc{F}(X),\mc{C}(X), \mc{K}(X), \mc{WC}(X), \mc{B}(X)$, and $\mc{P}(X)$ represent the set of all nonempty finite, closed and convex,  compact,  weakly compact, closed and bounded, and power set of $X$, respectively. It is assumed that the real line corresponds to the scalar field for the spaces. For a nonempty subset $B$ of $X$, we consider an ordered tuple $(t_b)_{b\in B}$, indexed by the set $B$ itself. We allow the repetitions, and by the well-ordering principle, such an ordering exists always.
We consider coordinate-wise ordering on a subclass of $\Pi_{b\in B}\mb{R}_{\geq 0}$, viz. $\ell_\iy(B)=\{\varphi:B\ra\mb{R}_{\geq 0}: \sup_{b\in B}\varphi(b)<\iy\}$, defined as: for $\varphi_1,\varphi_2\in\ell_\iy(B)$, $\varphi_1\leq\varphi_2$ if and only if $\varphi_1(b)\leq\varphi_2(b)$ for all $b\in B$.
For a subset $B$ of $X$, we consider $\ell_\iy(B)$ to be endowed with the supremum norm.

\bdfn
For a subset $B$ of $X$, and a function $f:\ell_\iy(B)\ra\mb{R}_{\geq 0}$, we call:
\bla
\item $f$ is {\it monotone} if for $\varphi_1,\varphi_2\in\ell_\iy(B)$, $\vp_1\leq\vp_2$ implies $f(\vp_1)\leq f(\vp_2)$.
\item $f$ is {\it strictly monotone} if for $\varphi_1,\varphi_2\in\ell_\iy(B)$, $\vp_1\leq\vp_2$ and $\vp_1\neq \vp_2$ imply $f(\vp_1)< f(\vp_2)$.
\item $f$ is {\it coercive} if $f(\vp)\ra\iy$ as $\|\vp\|_\iy\ra\iy$.
\el
\edfn

\bdfn\label{D1}
	\bla
	\item Let $Y$ be a subspace of $X$, $F\in\mc{B}(X)$, and a function $f: \ell_\iy(F)\ra\mb{R}_{\geq 0}$. We define $r_f(x,F)=f(\left(\|x-a\|\right)_{a\in F})$ and $\tr{rad}_Y^f(F)=\inf_{y\in Y} r_f(y,F)$ for $x\in X$. The collection of all points in $Y$ where the infimum that defines $\tr{rad}_Y^f(F)$ is achieving is referred to as the {\it restricted $f$-centers of $F$ in $Y$} and is denoted by $\tr{Cent}^f_Y(F)$. 
 That is, $\tr{Cent}^f_Y(F)=\{y\in Y:r_f(y,F)=\tr{rad}_Y^f(F)\}$.
	When $Y=X$, then the restricted $f$-centers are called the $f$-centers of $F$ in $X$. If the set of $f$-centers (restricted $f$-centers) is nonempty, we say that $X (Y )$ admits $f$-centers (restricted $f$-centers) for $F$.
	\item When $f$ is of the form $f(a)=\sup_{t\in F}\rho(t)a(t)$, where $\rho=(\rho(t))_{t\in F}\in\ell_\iy(F)$, then $r_f(x,F), \tr{rad}_Y^f(F)$, and $\tr{Cent}_Y^f(F)$ are rewritten as $r_\rho(x,F), \tr{rad}_Y^\rho(F)$, and $\tr{Cent}_Y^\rho(F)$, respectively. In this case, we refer to the $f$-centers as the {\it restricted weighted Chebyshev centers}. In this case, we call $Y$ admits restricted weighted Chebyshev center for $F$ for weights $(\rho(t))_{t\in F}$. When $\rho(t)=1$ for all $t$, we denote the previous quantities by $r(x,F), \tr{rad}_Y(F)$ and $\tr{Cent}_Y(F)$ respectively.
	\item For a subspace $Y$ of $X$, $\mf{F}\ci \mc{B}(X)$ and a family of functions $\mathscr{F}$ consisting of $f:\ell_\iy(F)\ra\mb{R}_{\geq 0}$ for all $F\in\mf{F}$,
	the triplet $(X,Y,\mf{F})$ is said to have the {\it restricted $\mr{F}$-center property} ( {\bf $\mr{F}$-rcp} in short) if for all $F\in\mf{F}$ and $f\in\mr{F}$, we have $\tr{Cent}_Y^f(F)\neq\es$.
	\item For a subspace $Y$ of $X$ and $\mf{F}\ci \mc{B}(X)$, the triplet $(X,Y,\mf{F})$ is said to have {\it restricted weighted Chebyshev center property} ({\bf wrcp} in short) if for all $F\in\mf{F}$ and bounded weights $\rho:F\ra \mb{R}_{\geq 0}$, $\tr{Cent}_Y^\rho(F)\neq\es$.
	\el
\edfn

We state the following result based on \cite{LV}.
\begin{thm}\label{6}
For a Banach space $X$ and any $n$-tuple $(a_1,\cdots,a_n)\in \mc{F}(X)$, the following are equivalent:
\bla
\item If $r_1,\ldots,r_n>0$ then $\cap_{i=1}^nB_{X^{**}}(a_i,r_i)\neq\es$ implies $\cap_{i=1}^nB_X(a_i,r_i)\neq\es$.
\item $X$ admits weighted Chebyshev centers for $(a_1,\cdots,a_n)$ for all weights $\rho_1,\cdots,\rho_n>0$.
\item $X$ admits $f$-centers for $(a_1,\cdots,a_n)$ for each continuous monotone coercive function $f$ on $\mathbb{R}^n_{\geq0}$.
\el
\end{thm}

A Banach space {\it $X$ belongs to the class $(GC)$} if  for every ordered $n$-tuple $(a_1,\ldots,a_n)\in \mc{F}(X)$ it satisfies any one of the equivalent conditions stated in Theorem~\ref{6}. It is denoted by $X\in (GC)$.

Bandyopadhyay and Rao in \cite{PB} introduced the notion of {\it central subspace}, which is a generalization of class $(GC)$ given in \cite{LV}. 

\bdfn\label{D3}
A subspace $Y$ of $X$ is said to be {\it a central subspace of $X$} if for any finite family of balls in $X$ with centers in $Y$ which intersect in $X$, also intersect in $Y$.
\edfn

Bandyopadhyay and Dutta in \cite{SD} further generalized the notion of central subspace to {\it $\mc{A}-C$-subspace}, for a subfamily $\mc{A}$ of $\mc{P}(X)$.

\bdfn\cite{SD}\label{D2}
Let $Y$ be a subspace of Banach space $X$  and $\mc{A}$ be a family of subsets of $Y$. 
\bla
\item $Y$ is said to be an {\it almost $\mc{A}-C$-subspace of $X$} if for all $x\in X$, $A\in\mc{A}$ and $\e>0$, there exists $y\in Y$ such that $\|y-a\|\leq\|x-a\|+\e$ for all $a\in A$.
\item Y is said to be a {\it $\mc{A}-C$-subspace of $X$} if we can take $\e=0$ in $(a)$.
\item If $\mc{A}$ is a family of subsets of $X$, then it is referred to that {\it $X$ has (almost) $\mc{A}$-IP}, if $X$ is an (almost) $\mc{A}-C$-subspace of $X^{**}$.
\el
\edfn

Note, $X\in (GC)$ if it has $\mc{F}(X)$-IP. Moreover,
$\mc{F}(Y)-C$-subspaces are called central subspaces in \cite{PB}.  

\bdfn
\bla
\item \cite{AL} A subspace $Y$ of $X$ is said to have {\it $n.X.I.P.$ in $X$} if for any $n$ closed balls $\{B_X(a_i,r_i)\}_{i=1}^n$ having centers in $Y$ and $\cap_{i=1}^nB_X(a_i,r_i)\neq\es$, then $\cap_{i=1}^nB_Y(a_i,r_i+\e)\neq\es$ for all $\e>0$.
\item \cite{AL} A subspace $Y$ of $X$ is said to be an {\it ideal} of $X$ if there exists a projection $P:X^*\ra X^*$ such that $\|P\|=1$ and $\ker (P)=Y^\perp$.
\el
\edfn

\begin{Rem}\label{R1}
	\bla
\item The subspaces $Y$ which have $n.X.I.P.$ for all $n$ are precisely those which are almost $\mc{F}(Y)-C$-subspaces of $X$. In general, $n.X.I.P.$ does not imply $k.X.I.P.$ for $k>n$, although under certain assumptions, $n.X.I.P.$ for all $n$ is equivalent to $3.X.I.P.$. 
\item A subspace $Y$ of $X$ which is known to be an ideal in $X$ also satisfies $n.X.I.P.$ (see \cite[Proposition~3.2]{AL}). 
\item A subspace that is a range of a norm-$1$ projection is clearly a central subspace and also an ideal. 
   \el
\end{Rem}

For a subspace $Y$ of $X$, we introduce the following notions here.

\bdfn
A subspace $Y$ of $X$ is said to have {\it restricted $n.X.I.P.$ in $X$} (be a {\it restricted central subspace of $X$}) if for any $n$ closed balls $\{B_X(a_i,r)\}_{i=1}^n$ having centers in $Y$ and $\cap_{i=1}^nB_X(a_i,r)\neq\es$, then $\cap_{i=1}^nB_Y(a_i,r+\e)\neq\es$ for all $\e>0$ ($\cap_{i=1}^n B_Y(a_i,r)\neq\es$).
\edfn

For a Banach space $X$, and a finite measure space $(\Om,\Si,\mu)$ and $1\leq p<\iy$, $L_p(\Si,X)$ represents the space of all {\it $p$-Bochner integrable functions}, which are precisely $f:\Omega\to X$ strongly measurable and $\int_\Omega \|f(t)\|^pd\mu (t)<\iy$. $\|f\|_p:=\left(\int_\Omega \|f(t)\|^pd\mu (t)\right)^{1/p}$ defines a norm on $L_p(\Si,X)$, which makes $L_p(\Si,X)$ a Banach space. $L_\iy(\Si,X)$ represents the set of all {\it essentially bounded functions} $f:\Omega\ra X$, which are strongly measurable. \cite[Ch. 2]{D} is a standard reference for these spaces and all the properties used in this article.

Let us recall from \cite[Ch. 5]{D} that for a  sub $\si$-algebra $\Si'\ci \Si$, by considering $L_p(\Si',X)$ as a subspace of $L_p(\Si,X)$, $1\leq p<\iy$, the  {\it conditional expectation operator} is a mapping  $E:L_p(\Si,X)\ra L_p(\Si',X)$ such that $E(f)=g$, where $\int_Bfd\mu=\int_Bgd\mu$ for all $B\in\Si'$ indicates a linear projection of norm-$1$. 

\bdfn\cite{H}
\bla
\item A bounded linear projection $P:X\ra X$ is said to be an {\it $L$-projection} if 
$\|x\|=\|Px\|+\|x-Px\|$ for all $x\in X$.

\item A Banach space $X$ is said to be {\it $L$-embedded} if $X$, under its canonical image in $X^{**}$, is the range of an $L$-projection on $X^{**}$. 

\item A closed subspace $J\ci X$ is an {\it $M$-ideal} in $X$ if $J^\perp$ is the range of an $L$-projection on $X^*$(\cite{H}). 
\el
\edfn

Reference \cite[Ch. 4]{H} provides examples and other properties of $L$-embedded spaces. We call a subspace $Y$ of $X$ is {\it $1$-complemented} if $Y$ is a range of a norm-$1$ projection from $X$.
If $Y$ is a $1$-complemented subspace of an $L$-embedded space $X$, then $Y$ is also $L$-embedded (see \cite[Theorem~IV.1.5.]{H}). For a sub $\si$-algebra $\Si'\ci\Si$, due to the conditional expectation $E:L_1(\Si,X)  \ra L_1(\Si',X)$, $L_1(\Si',X)$ is $L$-embedded if $L_1(\Si,X)$ is so.

A Banach space $X$ is said to be a {\it Lindenstrauss space} (or {\it $L_1$-predual}) if $X^*\cong L_1(\mu)$ for some measure $\mu$. In \cite{JL1}, Lindenstrauss characterizes these spaces as Banach spaces, where any collection of pairwise intersecting closed balls whose centers form a compact set has a nonempty intersection.

An ideal of a Lindenstrauss space is also a Lindenstrauss space and, hence, a central subspace (see \cite[Lemma~10, Theorem~15]{CR}).

In \cite{LV}, Vesel\'{y} derived that for spaces $X$ with the Radon-Nikod\'{y}m-property (RNP in short) and $1$-complemented in its bidual, $L_p(\mu,X)\in (GC)$, for $1\leq p<\iy$. Additionally, if $X$ is a dual space that is strictly convex and has $(w^*K)$, then $C_b(T,X)\in (GC)$ for a Hausdorff space $T$.

We refer to the articles \cite{SD, LV} for various  examples of the spaces that are  discussed in this paper.

\subsection{Observations:~} 
In this subsection, we give an overview of our observations.
Suppose that, $Y$ has $n.X.I.P.$ in $X$ and $F\in\mc{F}(Y)$, where $card (F)=n$ and $X$ admits weighted Chebyshev centers for $F$. In section~2, it is demonstrated that $Y$ admits restricted $f$-centers for $F$, where $card (F)=n$, for all continuous, monotone functions $f:\mb{R}^n_{\geq0}\ra\mb{R}_{\geq0}$,  if and only if for any collection of $n$ closed balls with centers in $F$ having a nonempty intersection in $X$ also intersect in $Y$. We obtain a similar characterization to that \cite[Theorem~2.7]{LV} for the family of closed and bounded subsets of $Y$.

Problem~\ref{Q1} is answered in Theorem~\ref{T13}, which extends a 
characterization to the notion $n.X.I.P.$ in Banach spaces. Consequently, we obtain various characterizations for central subspaces in Banach spaces, as stated in Theorem~\ref{T3} and Theorem~\ref{T9}.
It is easy to observe that if $Y$ is a subspace of $X$ and if there exists $P:X\ra Y$ onto, where $P(\la x+y)=\la P(x)+y$, for $x\in X, y\in Y$ and scalar $\la$, $\|P(x)\|\leq \|x\|$ for all $x$, a {\it quasi-linear projection} $P:X\ra X$, then for any $F\in\mc{B}(Y)$, $\tr{rad}_Y(F)=\tr{rad}_X(F)$. Vesel\'{y}'s example in \cite[Pg.9]{LV} ensures there does not exist any quasi-linear projection $P:\ell_\iy\ra c_0$. We do not know the solution to the question posed in Problem~\ref{Q1} for closed bounded subsets, with the exception of some trivial circumstances, such as when $Y$ is $1$-complemented in $X$ or, more generally, a range of a quasi-linear projection. The requirement for $Y$ being a central subspace can be more simply expressed when $Y\in (GC)$ is observed. 

Let us give an example of a space $X$ such that $L_1(\mu,X)\notin (GC)$. Let $f\in \ell_1$ be such that support of $f$ is infinite and $2\|f\|_\iy>\|f\|_1$. Now identify $f$ as a linear functional on $c_0$ and suppose that $X=\ker (f)$. Thus, $X\notin (GC)$ (see \cite{LV1}). Since $X$ is $1$-complemented in $L_1(\mu,X)$, $L_1(\mu,X)\notin (GC)$.

In section~3, we investigate the restricted $\mr{F}$-center property for various triplets $(X, Y, \mf{F})$, where $Y$ represents a subspace of $X$ and $\mf{F}\ci \mc{B}(X)$. It is evident that when a given subspace $Y\ci X$, $(X,Y,\mc{F}(X))$ possesses {\bf wrcp}, then $Y\in (GC)$. On the other hand, from Theorem~\ref{T12} one can conclude in the above instances that $Y$ is a central subspace of $X$ whenever $X$ admits $f$-centers for $F\in \mc{F}(Y)$, for all continuous, monotone functions $f:\ell_\iy(F)\ra\mb{R}_{\geq 0}$.

Methods for dealing with the aforementioned ideas, including $(GC)$, central subspace, $n.X.I.P.$, $\mc{A}-C$-subspace are adapted from \cite{SD,PB,T,LV1,LV}. In section~ 3, we use certain measure-theoretic tools from \cite{D} to derive our observations. 

\section{$\mc{A}-C$-subspaces of Banach spaces}\label{13}
The following is obtained by employing justifications similar to those given in \cite[Theorem~2.6]{LV}. For an $F\in \mc{F}(X)$, by $card (F)$, we mean the cardinality of the ordered tuple.

\begin{thm}\label{T10}
Let $Y$ be a subspace of $X$ and $n\in\mb{N}$. Suppose $Y$ has $n.X.I.P.$ in $X$, and $f:\mb{R}^n_{\geq0}\ra\mb{R}_{\geq0}$ be a continuous, monotone function. Then, for all $F\in\mc{F}(Y)$ with $card (F)=n$, we have $\emph{rad}_Y^f(F)=\emph{rad}_X^f(F)$.
\end{thm}

\begin{proof}
Suppose that $F=(z_1,\ldots, z_n)$ be an ordered $n$-tuple. Let ${\bf{t_0}}=(\|x-z_1\|,\ldots,\|x-z_n\|)$.
Choose $\e>0$ and let $\de>0$ be such that $|f({\bf{t}})-f({\bf{t_0}})|\leq\e$ whenever $d(\bf{t},\bf{t_0})\leq\de$. Without loss of generality, we may assume $\de<\e$.

Let $x\in X$ be such that, $f(\|x-z_1\|,\ldots,\|x-z_n\|)<\tr{rad}_X^f(F)+\de$. Since $Y$ has $n.X.I.P.$ in $X$, there exists $y\in \cap_{i=1}^nB_Y(z_i,\|x-z_i\|+\de)$.
\beqa
\mbox{Now~} \tr{rad}_Y^f(F) &\leq & f(\|y-z_1\|,\ldots,\|y-z_n\|)\\
                             &\leq & f(\|x-z_1\|+\de,\ldots,\|x-z_n\|+\de)\\
                             &\leq & f(\|x-z_1\|,\ldots,\|x-z_n\|)+\e \\
                             &<& \tr{rad}_X^f(F)+\de+\e<\tr{rad}_X^f(F)+2\e.
\eeqa
Since $\e> 0$ is arbitrary, the result follows.
\end{proof}

The proof of the following theorem follows in the same manner as the proof of  \cite[Theorem~2.7]{LV}.

\begin{thm}\label{T8}
Let $\mc{A}$ be one of the families $\mc{F}, \mc{K}$ and $\mc{B}$. Assume that $F\in \mc{A}(Y)$, where $F=(y_i)_{i\in F}$. Suppose that for each continuous, monotone function $f:\ell_\iy(F)\ra\mb{R}_{\geq0}$,  $\emph{rad}_Y^f(F)=\emph{rad}_X^f(F)$ and $X$ admits $f$-centers for $F$. Then \tfae.
\bla
\item If for $i\in F$, $r_i\in\mb{R}_{>0}$, $\cap_{i\in F}B_X(y_i,r_i)\neq\es$, then $\cap_{i\in F}B_Y(y_i,r_i)\neq\es$.
\item $Y$ admits restricted weighted Chebyshev centers for $F$ for all weights $\rho:F\ra\mb{R}_{>0}$.
\item $Y$ admits restricted $f$-centers for $F$ for all continuous, monotone functions $f:\ell_\iy(F)\ra \mb{R}_{\geq 0}$.
\el
\end{thm}
\begin{proof}
$(a)\Ra(c)$. Let $f:\ell_\iy(F)\ra\mb{R}_{\geq0}$ be a continuous, monotone function and $x\in\tr{Cent}_X^f(F)$. As $x\in \cap_{i\in F} B_X(y_i,\|x-y_i\|)$, there exists $y\in \cap_{i\in F} B_Y(y_i,\|x-y_i\|)$.
Since $\tr{rad}_Y^f(F)=\tr{rad}_X^f(F)$, it can be clearly understood that  $y\in\tr{Cent}_Y^f(F)$.

$(c)\Ra(b)$. This is obvious.

$(b)\Ra(a)$. Let $x\in\cap_{i\in F}B_X(y_i,r_i)$. Define $\rho_i=\frac{1}{r_i}$ for all $i\in F$. Accordingly, $r_\rho(x,F)\leq1$. Based on our assumption, we obtain $\tr{rad}_Y^\rho(F)=\tr{rad}_X^\rho(F)\leq r_\rho(x,F)\leq1$. Moreover, there exists $y\in Y$ such that $y\in\tr{Cent}_Y^\rho(F)$. Thus $\frac{1}{r_i}\|y-y_i\|\leq1$ for all $i\in F$. Thus $y\in\cap_{i\in F}B_Y(y_i.r_i)$.
\end{proof}

By varying $F\in \mc{A}(Y)$ in Theorem~\ref{T8} we obtain the following.

\begin{thm}\label{T12}
Let $\mc{A}$ be one of the families $\mc{F}, \mc{K}$ and $\mc{B}$. Let $Y$ be a subspace of $X$ and for all $F\in\mc{A}(Y)$, $\emph{rad}_Y^f(F)=\emph{rad}_X^f(F)$ and $X$ admit $f$-centers for $F$, for all continuous monotone $f:\ell_\iy(F)\ra \mb{R}_{\geq 0}$. Then \tfae.
\bla
\item $Y$ is a $\mc{A}-C$-subspace of $X$.
\item The triplet $(X,Y,\mc{A}(Y))$ has {\bf wrcp}.
\item $\emph{Cent}_Y^f(F)=\emph{
Cent}_X^f(F)\cap Y\neq\es$, for all $F\in\mc{A}(Y)$ and each continuous, monotone $f:\ell_\iy(F)\ra \mb{R}_{\geq 0}$.
\el
\end{thm}

We now focus on the family $\mc{F}(Y)$ for a subspace $Y$ of $X$. Theorem~\ref{T12} reduces to the characterizations of central subspaces by taking $\mc{A}=\mc{F}$. Some other characterizations of central subspaces, $(GC)$ are also obtained in the subsequent part of this section.

\begin{lemma}\label{T7}
Let $Y$ be a subspace of $X$. If $Y$ has restricted $n.X.I.P.$ in $X$, then $Y$ has $n.X.I.P.$ in $X$.
\end{lemma}
\begin{proof}
Let $\{B_X(y_i,r_i)\}_{i=1}^n$ be a collection of $n$ closed balls in $X$, where $y_1,\cdots,y_n\in Y$ and $r_1,\cdots,r_n>0$ such that $\cap_{i=1}^nB_X(y_i,r_i)\neq\es$.

Suppose there exists an $\e>0$ for which $\cap_{i=1}^nB_Y(y_i,r_i+\e)=\es$. Let $r\in\mathbb{R}$ be such that $r-\frac{\e}{2}>r_i$ for all $i=1,\cdots,n$. We now construct $n$ closed balls $\{B_X(w_i,r-\frac{\e}{2})\}_{i=1}^n$ with $w_1\cdots,w_n\in Y$ such that $B_X(w_i,r-\frac{\e}{2})\supseteq B_X(y_i,r_i)$ and $\cap_{i=1}^nB_Y(w_i,r)=\es$, which contradict our assumption. We now choose $w_i$ inductively.

We define $L_1=\cap_{i=2}^nB_Y(y_i,r_i+\e)$ and $L_2=B_Y(y_1,r_1+\e)$.

{\sc Case~1:} Suppose $L_1=\es$. Then we chose $w_1=y_1$ and obtain 
$B_Y(w_1,r)\cap\left(\cap_{i=2}^nB_Y(y_i,r_i+\e)\right)=\es$ and proceed further to obtain $w_2$.

{\sc Case~2:} Suppose that $L_1\neq\es$. Thus, $L_1\cap L_2=\es$ and $(L_1-y_1)\cap B_Y(0,r_1+\e)=\es$. Thus, we get $g\in Y^*$ such that $g(y-y_1)\geq 1\geq g(x)$ for all $y\in L_1$ and $x\in B_Y(0,r_1+\e)$. 

Now, $\|g\|=\sup_{y\in B_Y}g(y)=\frac{1}{r_1+\e}\sup_{y\in B_Y(0,r_1+\e)}g(y)\leq\frac{1}{r_1+\e}$. 

Choose $\de>0$ such that $\de<\frac{1}{r-\frac{\e}{2}-r_1}\big(1-\frac{r_1+\frac{\e}{2}}{r_1+\e}\big)$. Let $z\in S_Y$ be such that $g(z)\leq -\|g\|+\de$. Let $w_1=y_1+(r-\frac{\e}{2}-r_1)z\in Y$. 

Let $x\in B_X(y_1,r_1)$. Thus,
\beqa
\|x-w_1\|&\leq&\|x-y_1\|+ (r-\frac{\e}{2}-r_1)\|z\|\\
&\leq& r_1+(r-\frac{\e}{2}-r_1)\\
&=&r-\frac{\e}{2}.
\eeqa
Thus, $x\in B_X(w_1,r-\frac{\e}{2})$ and hence $B_X(y_1,r_1)\ci B_X(w_1,r-\frac{\e}{2})$.

Let $y\in B_Y(w_1,r)$. Thus,
\beqa
g(y-y_1)&=&g(y-w_1)+g(w_1-y_1)\\
&\leq&\|g\|r+(r-\frac{\e}{2}-r_1)g(z)\\
&\leq&\|g\|r+(r-\frac{\e}{2}-r_1)(-\|g\|+\de)\\
&\leq&\frac{r_1+\frac{\e}{2}}{r_1+\e}+\de(r-\frac{\e}{2}-r_1)<1.
\eeqa
Thus, for all $y\in B_Y(w_1,r)$, $y\notin L_1$. Hence, $L_1\cap B_Y(w_1,r)=\es$.

Now, suppose that we have $w_i$ for $i\leq j(<n)$ such that, 
\beqa
B_X(w_i,r-\frac{\e}{2})\supseteq B_X(y_i,r_i) ~\text{for}~ i\leq j~\text{and}\\
\cap_{i=1}^jB_Y(w_i,r)\cap\left(\cap_{i=j+1}^nB_Y(y_i,r_i+\e)\right)=\es.
\eeqa
Let us define $K_1=\cap_{i=1}^jB_Y(w_i,r)\cap\left(\cap_{i=j+2}^nB_Y(y_i,r_i+\e)\right)$ and $K_2=B_Y(y_{j+1},r_{j+1}+\e)$. 

We follow the similar techniques that are used in order to obtain $w_{1}$.
The two sets $K_1$ and $K_2$ play similar roles as the sets $L_1$ and $L_2$ like before, and hence we get $w_{j+1}$. This completes the induction, and the proof follows.
\end{proof}

\begin{Cor}\label{C2}
Let $Y$ be a subspace of $X$ such that $\emph{rad}_Y(F)=\emph{rad}_X(F)$ for all $F\in\mc{F}(Y)$. Then for all $n\in\mb{N}$, $Y$ possesses $n.X.I.P.$ in $X$.
\end{Cor}
\begin{proof}
Let $y_1,\ldots,y_n\in Y$ and $\{B_X(y_i,r)\}_{i=1}^n$ be a finite family of closed balls in $X$, where $\cap_{i=1}^nB_X(y_i,r)\neq\es$ in $X$. Hence, $\tr{rad}_Y(F)=\tr{rad}_X(F)\leq r$. As $\cap_{i=1}^nB_Y(y_i,\tr{rad}_Y(F)+\e)\neq\es$ holds for all $\e>0$, we have that the balls $\{B_X(y_i,r)\}_{i=1}^n$ almost intersect in $Y$. The result now follows from Lemma~\ref{T7}.
\end{proof}

Combining Theorem~\ref{T10}, Lemma~\ref{T7}, and Corollary~\ref{C2}, we obtain the following.

\begin{thm}\label{T13}
Let $Y$ be a subspace of $X$. Then \tfae.
\bla
\item $Y$ has $n.X.I.P.$ in $X$, for all $n$.
\item $Y$ has restricted $n.X.I.P.$ in $X$, for all $n$.
\item $\emph{rad}_Y(F)=\emph{rad}_X(F)$, for all $F\in\mc{F}(Y)$.
\item $\emph{rad}_Y^f(F)=\emph{rad}_X^f(F)$, for all $F\in\mc{F}(Y)$, for all continuous, monotone functions $f:\mb{R}^n_{\geq 0}\ra \mb{R}_{\geq 0}$, $n\geq 1$.
\el
\end{thm}

We now derive a few characterizations of central subspaces of Banach spaces.

\begin{thm}\label{T3}
Let $Y$ be a subspace of $X$ and $Y\in (GC)$. Thus, $Y$ is a restricted central subspace of $X$ if and only if $Y$ is a central subspace of $X$.
\end{thm}
\begin{proof}
Let $\{B_X(y_i,r_i)\}_{i=1}^n$ be a collection of $n$ closed balls in $X$, where $y_1,\cdots,y_n\in Y$ and $r_1,\cdots,r_n>0$ such that $\cap_{i=1}^nB_X(y_i,r_i)\neq\es$.
	
{\sc Claim:} $\cap_{i=1}^nB_Y(y_i,r_i)\neq\es$.
	
By \cite[Proposition 2.9]{PB}, it can be sufficiently shown that $\cap_{i=1}^n B_Y(y_i,r_i+\e)\neq\es$ for all $\e>0$. Accordingly, from Lemma \ref{T7}, we conclude the proof.
\end{proof}

\begin{thm}\label{T9}
Let $Y$ be a subspace of $X$ and $X$ admit weighted Chebyshev centers for all finite subsets of $Y$. Suppose that $Y$ has $n.X.I.P.$ in $X$ for all $n$. Then, $Y$ is a central subspace of $X$ if and only if $Y\in (GC)$.
\end{thm}
\begin{proof}
Suppose that $Y$ is a central subspace of $X$. Thus, from Theorems \ref{T10} and \ref{T8}, we obtain $(X,Y,\mc{F}(Y))$ has {\bf wrcp}.
	
Conversely, assume that $Y\in (GC)$. Hence, by Theorem \ref{T3}, it is sufficient to consider balls with centers in $Y$ of equal radii. Let $\{B_X(y_i,r)\}_{i=1}^n$ be a collection of $n$ balls with $y_1,\cdots,y_n\in Y$, $r>0$ such that $\cap_{i=1}^nB_X(y_i,r)\neq\es$. Let $F=(y_1,\cdots,y_n)$. Thus, $\tr{rad}_X(F)\leq r$ and by our assumption, $\tr{rad}_Y(F)\leq r$.  Owing to the fact that $(X,Y,\mc{F}(Y))$ has {\bf wrcp}, $\cap_{i=1}^nB_Y(y_i,\tr{rad}_Y(F))\neq\es$ and consequently $\cap_{i=1}^nB_Y(y_i,r)\neq\es$.
\end{proof}

\begin{thm}
Let $\mc{A}$ be any of the families $\mc{F},\mc{K},\mc{B}$, or $\mc{P}$, $P$ be a projection on $X$ of norm-$1$, $Y=\ker (P)$,  and $Z\ci P(X)$ be a subspace. If $Y+Z$ is an (almost) $\mc{A}(Y+Z)-C$-subspace of $X$, then $Z$ is an (almost) $\mc{A}(Z)-C$-subspace of $P(X)$.
\end{thm}
\begin{proof}
Let $A\in\mc{A}(Z)$. Suppose $\underset{a\in A}{\cap}B_X(a,r_a)\cap P(X)\neq\es$. Thus, $\underset{a\in A}{\cap}B_X(a,r_a)\cap X\neq\es$. Based on our assumption, $\underset{a\in A}{\cap}B_X(a,r_a)\cap (Y+Z)\neq\es$. Let $y_0+z_0\in \underset{a\in A}{\cap}B_X(a,r_a)\cap (Y+Z)\neq\es$. Accordingly, for all $a\in A$,
   \beqa
   \|z_0-a\|&=&\|P(z_0-y_0-a)\|
   \leq\|z_0-y_0-a\|
   \leq r_a.
   \eeqa
This proves $Z$ is an $\mc{A}(Z)-C$-subspace of $P(X)$.
The remaining follows in a similar way as stated above.
\end{proof}

\begin{Pro}
Let $\mc{A}$ be any of the families $\mc{F},\mc{K},\mc{B}$, or $\mc{P}$. Let $Z\ci Y\ci X$ represent subspaces of $X$ such that $Y$ is an M-ideal in $X$. If $Z$ is an (almost) $\mc{A}(Z)-C$-subspace of $Y$ and $Y$ has (almost) $\mc{A}(Y)$-IP, then $Z$ is an (almost) $\mc{A}(Z)-C$-subspace of $X$. 
\end{Pro}
\begin{proof}
Let $A\in\mc{A}(Z)$ and $\{B_X(a,r_a)\}_{a\in A}$ be a collection of closed balls, where  $r_a>0$ for all $a\in A$, such that $\cap_{a\in A}B_X(a,r_a)\neq\es$. Since $Y$ is an $M$-ideal in $X$, we obtain $X^{**}=Y^{\perp\perp}\oplus_\iy N^\perp$ for a subspace $N$ of $X^*$. If $x\in \cap_{a\in A}B_X(a,r_a)$, then it can be assumed that $x=y+z$, where $y\in Y^{\perp\perp}$ and $z\in N^\perp$. Now $\max\{\|y-a\|,\|z\|\}\leq r_a$ for all $a\in A$. Since $Y$ has $\mc{A}(Y)$-IP, there exists $y_0\in Y$ such that $\|y_0-a\|\leq r_a$ for all $a\in A$. Thus, based on our assumption, $\cap_{a\in A} B_Z(a,r_a)\neq\es$.
\end{proof}

Theorem ~\ref{T5} and Corollary \ref{C1} indicate the consequences of our previous  observations and are derived from the techniques employed in \cite[Theorem 2.1, Corollary 2.2]{DY1}. We now use the well-known Michael selection theorem (\cite[Theorem~3.2$^{''}$]{EM}) several times to obtain our observations.

\begin{thm}\label{T5}
Let $X$ be a Lindenstrauss space, $K$ and $S$ compact Hausdorff spaces, and $\psi:K\ra S$ a continuous onto map. Let $\psi^*:C(S,X)\ra C(K,X)$ be a continuous isometric embedding expressed as $\psi^*f=f\psi$. Accordingly, $\psi^*C(S,X)$ is a $\mc{K}(\psi^*C(S,X))-C$-subspace of $C(K,X)$.
\end{thm}
\begin{proof}
Let $A\in\mc{K}(\psi^*C(S,X))$ and $B=\{f\in C(S,X):\psi^*f\in A\}$. Suppose $\cap_{ f\in B}B(\psi^*f,r_f)\neq\es$, where $r_f>0$ for all $f\in B$. We define a multivalued map $F:S\ra \mc{C}(X)$ by $F(y)=\cap_{f\in B}B(f(y),r_f)$ for all $y\in S$. Evidently, $F(y)$ is closed, convex, and $F(y)\neq\es$ for all $y\in S$.

{\sc Claim :} $F$ is lower semicontinuous.

Let $G\ci X$ be open and $y_0\in \{y\in S:F(y)\cap G\neq\es\}$. Let $a\in F(y_0)\cap G$. Thus, we have $a\in\cap_{f\in B}B(f(y_0),r_f)$ and $B(a,\e)\ci G$ for some $\e>0$. Accordingly, there exists an open neighbourhood $N$ of $y_0$ such that $\|f(y)-f(y_0)\|<\e$ for all $y\in N$ and $f\in B$. Now for all $y\in N$, we have $B(f(y),r_f)\cap B(a,\e)\neq\es$ for all $f\in B$ and since $X$ is a Lindenstrauss space, $F(y)\cap G\neq\es$. Thus, $N\ci \{y\in S: F(y)\cap G\neq\es\}$ and $F$ is lower semicontinuous. Thus, in accordance with Michael's selection theorem, there exists a continuous $h:S\ra X$ such that $h(y)\in F(y)$ for all $y\in S$. It is evident that $\psi^*h\in \cap_{f\in B}B(\psi^*f,r_f)\cap\psi^*C(S,X)$.
 \end{proof}
 
 \begin{Cor}\label{C1}
Let $S,K,\psi$ and $X$ be as in Theorem \ref{T5}. Furthermore, fix $y_0\in S$ and let $M=\{\psi^*f:f\in C(S,X)~ \tr{and}~ f(y_0)=0\}$. Then, $M$ is a $\mc{K}(M)-C$-subspace of $C(K,X)$. \end{Cor}
\begin{proof}
Let $A\in \mc{K}(M)$ and $B=\{f\in C(S,X):\psi^*f\in A\}$. Suppose $\cap_{f\in B}B(\psi^*f,r_f)\neq\es$, where $r_f>0$ for all $f\in B$. Define $F:S\ra \mc{C}(X)$ by $F(y)=\cap_{f\in B}B(f(y),r_f)$ for all $y\in S$. As in the proof of Theorem \ref{T5}, we obtain $F(y)\neq\es$ for all $y\in S$. Now we define $F_0:S\ra \mc{C}(X)$ by:  
\begin{equation*}
F_0(y)=\begin{cases}
F(y) \quad &\text{if} \, y \neq y_0 \\
  \{0\} \quad &\text{if} \, y = y_0 \\
\end{cases}
\end{equation*}
 Clearly, $F_0(y)$ is closed, convex and $F_0(y)\neq\es$ for all $y\in S$. Similarly, as done in the proof of Theorem \ref{T5}, it can be shown that $F_0$ is lower semicontinuous. Thus, by Michael's selection theorem, we obtain $h:S\ra X$ such that $h(y)\in F_0(y)$ for all $y\in S$. Thus, $\psi^*h\in\cap_{f\in B}B(\psi^*f,r_f)\cap M$.
\end{proof}

Proceeding similarly to the proof of Theorem \ref{T5}, the following can be obtained.

\begin{Cor}
\bla
\item Let $X$ be a Lindenstrauss space, $T$ compact Hausdorff space and $S$ be a closed subset of $T$. If $M=\{f\in C(T,X):f|_S=0\}$, then $M$ is a $\mc{K}(M)-C$-subspace of $C(T,X)$.
\item Let $X$ be a Lindenstrauss space and $T$ be a compact Hausdorff space. Then, $C(T,X)$ is a $\mc{K}(C(T,X))-C$-subspace of $\ell_\iy(T,X)$.
\el
\end{Cor}

Our following observation shows that the characteristic of being a central subspace is stable in the spaces of continuous functions. Let us recall the manner in which the Lindenstrauss spaces are characterized in the introduction.

\begin{thm}\label{T6}
Let $\mc{A}$ be any of the families $\mc{F},\mc{K}$. Let $X$ be a Lindenstrauss space, $Y$ a $\mc{A}(Y)-C$-subspace of $X$, and $T$ a compact Hausdorff space. Then, $C(T,Y)$ is a $\mc{A}(C(T,Y))-C$-subspace of $\ell_\iy(T,X)$.
\end{thm}

\begin{proof}
We follow the similar techniques that are used in Theorem~\ref{T5}.

Let $A\in\mc{A}(C(T,Y))$ and  $\cap_{f\in A}B(f,r_f)\cap \ell_\iy(T,X)\neq\es$, where $r_f>0$ for all $f\in A$. Define $F:T\ra \mc{C}(Y)$ by $F(t)=\cap_{f\in A}B(f(t),r_f)\cap Y$. Now since for all $t\in T$, $\cap_{f\in A}B(f(t),r_f)\cap X\neq\es$, based on our assumption, $\cap_{f\in A}B(f(t),r_f)\cap Y\neq\es$ for all $t\in T$. Clearly, $F(t)$ is closed and convex for all $t\in T$.

By using similar arguments stated in Theorem~\ref{T5} it follows that $F$ is lower semicontinuous.

By Michael's selection theorem, we obtain $h:T\ra Y$ such that $h(t)\in F(t)$ for all $t\in T$. This completes the proof.
\end{proof}

We end this section by providing a straightforward application of the conditional expectation projection $E:L_p(\Si,X)\ra L_p(\Si',X)$, for $1\leq p<\iy$. 
In the following $(\Omega,\Si,\mu)$ denotes a finite measure space and $\Si^\prime$ is a sub $\si$-algebra of $\Si$.

\begin{Pro}\label{P3}
Let $Y$ be a subspace of $X$ and $\mc{A}$ be any of the families $\mc{F},\mc{K},\mc{B}$, or $\mc{P}$. Let $L_1(\Si,Y)$ be a $\mc{A}(L_1(\Si,Y))-C$-subspace of $L_1(\Si,X)$. Consequently, $L_1(\Si',Y)$ is an $\mc{A}(L_1(\Si',Y))-C$-subspace of $L_1(\Si',X)$.
\end{Pro}

\begin{Rem}
In section~3, we encounter cases in Theorem \ref{T2} and \ref{T11} when $L_1(\Si,Y)$ is a $\mc{A}(L_1(\Si,Y))-C$-subspace of $L_1(\Si,X)$, for subspaces $Y$ of $X$.
\end{Rem}

\section{Weighted restricted Chebyshev centers in spaces of vector-valued functions} 

First, we note that the following may be established by the arguments given in \cite[Theorem~2]{DA}. 

\begin{thm}
Let $X$ be a Banach space which is uniformly convex and let $T$ be a topological space. The Banach space $C_b(T, X)$ admits weighted Chebyshev centers for all closed bounded subsets of $C_b(T, X)$ and bounded weights $\rho$.
\end{thm}

The following result demonstrates that in the spaces of Bochner integrable functions $L_p(\Si,X)$, where $(\Omega,\Si,\mu)$ is a finite measure space, the property $\mr{F}$-{\bf \tr{rcp}} is stable for a particular class of functions $\mr{F}$.
For $n\in\mb{N}$ and $1\leq p<\iy$, $\mr{F}_p$ denotes a countable family of functions viz. $(f_p^n)_{n=1}^\iy$, where $f_p^n:\mb{R}^n_{\geq0}\ra\mb{R}_{\geq0}$ is defined by $f_p^n(\al_1,\cdots,\al_n)=\left(\sum_{i=1}^n|\al_i|^p\right)^{\frac{1}{p}}$.

\begin{thm}
Consider the measure space $(\Omega, \Sigma, \mu)$ as stated above. Let $Y$ be a separable subspace of $X$ and consider the corresponding function spaces $L_p(\Si, Y)\ci L_p(\Si, X)$. Let $\mr{F}_p$ be the family of functions as stated above, then:
\bla
\item  $(X,Y,\mc{F}(X))$ has \emph{\bf wrcp} if and only if $(L_{\iy}(\Si,X),L_\iy(\Si,Y),\mc{F}(L_\iy(\Si,X)))$ has \emph{\bf wrcp}.
\item for $1\leq p<\iy$, $(X,Y,\mc{F}(X))$ has \emph{\bf $\mr{F}_p$-rcp} if and only if $(L_p(\Si,X),L_p(\Si,Y),\mc{F}(L_p(\Si,X)))$ has \emph{\bf $\mr{F}_p$-rcp}.
\el
\end{thm}
\begin{proof}
$(a)$. Let $F\in\mc{F}(L_\iy(\Si,X))$. We prove the result when $card(F)=2$ because new ideas are not involved for higher values of $card(F)$. Let $F=\{f_1,f_2\}$ and $\rho=(\rho_1,\rho_2)$ be the corresponding weights. Now, there exists a measurable subset $E\ci \Omega$ such that $\mu(E)=0$ and $f_1,f_2$ are bounded on $\Omega\sm E$. Define for all $t\in \Omega\sm E$, $F_t=\{f_1(t),f_2(t)\}$. Let $G=\{(t,y)\in (\Omega\sm E)\times Y:r_\rho(y,F_t)=\textrm{rad}_Y^\rho(F_t)\}$. Since $(X,Y,\mc{F}(X))$ has {\bf wrcp}, the projection of $G$ on $\Omega\sm E$ is $\Omega\sm E$.

Suppose $(y_n)$ is dense in $Y$. Then, $G=\cap_{n\in\mathbb{N}}\{(t,y)\in (\Omega\sm E)\times Y: r_\rho(y,F_t)\leq r_\rho(y_n,F_t)\}$. Since all the involved functions are measurable, $G$ is a measurable set. Thus, as a consequence of the von Neumann selection theorem \cite[Theorem~7.2]{TP}, we obtain a measurable function $g_0:(\Omega\sm E)\ra Y$ such that $(t,g_0(t))\in G$ for $\mu$-a.e $t$. Then, $g_0(t)\in\tr{Cent}_Y^\rho(F_t)$ for $\mu$-a.e. $t\in \Omega\sm E$. Clearly, $g_0\in L_\iy(\Si,Y)$. For $g\in L_\iy(\Si,Y)$, we obtain
\beqa
r_\rho(g,F)&=&\underset{i=1,2}{\max}~\underset{t\in \Omega}{\textrm{ess~sup}}\rho_i\|g(t)-f_i(t)\|\\
&=&\underset{t\in \Omega}{\textrm{ess~sup}}~\underset{i=1,2}{\max}\rho_i\|g(t)-f_i(t)\|\\
&\geq&\underset{t\in \Omega}{\textrm{ess~sup}}~\underset{i=1,2}{\max}\rho_i\|g_0(t)-f_i(t)\|\\
&=&\underset{i=1,2}{\max}~\underset{t\in \Omega}{\textrm{ess~sup}}\rho_i\|g_0(t)-f_i(t)\|\\
&=&r_\rho(g_0,F).
\eeqa
Hence, $g_0\in\textrm{Cent}^\rho_{L_\iy(\Si,Y)}(F)$.

$(b)$. Let $F\in\mc{F}(L_p(\Si,X))$. We prove the result when $card(F)=2$, as no new ideas are involved for higher values of $card(F)$. Let $F=\{f_1,f_2\}$ and $f_p^2:\mb{R}_{\geq0}^2\ra\mb{R}_{\geq0}$ represent the function in $\mr{F}_p$ as stated above. Accordingly, there exists a measurable subset $E\ci \Omega$ such that $\mu(E)=0$ and $f_1,f_2$ are finite valued on $\Omega\sm E$. Define for all $t\in \Omega\sm E$, $F_t=\{f_1(t),f_2(t)\}$. Let $G=\{(t,y)\in (\Omega\sm E)\times Y:r_{f_p^2}(y,F_t)=\textrm{rad}_Y^{f_p^2}(F_t)\}$. Since $(X,Y,\mc{F}(X))$ has {\bf $\mr{F}_p$-rcp}, the projection of $G$ on $\Omega\sm E$ is $\Omega\sm E$.

Thus, in a similar manner as the proof of $(a)$, we obtain a measurable function $g_0:(\Omega\sm E)\ra Y$ such that $(t,g_0(t))\in G$ $\mu$-a.e. $t$. Clearly, $g_0\in L_p(\Si,Y)$. It is easy to see that $g_0\in\tr{Cent}_{L_p(\Si,Y)}^{f_p^2}(F)$.
\end{proof}

Our next few results are derived under an assumption that the function $r_f(.,F)$ is lower semicontinuous w.r.t. a suitable topology on $X$, for a given bounded set $F\in\mc{B}(X)$, which is now stated in the following remark. For an arbitrary set $\Gamma$, we define $\ell_1(\Gamma):=\{\vartheta:\Ga\ra\mb{R}_{\geq 0}:\sum_\ga \vartheta (\ga)<\iy\}$.

\begin{Rem}\label{R2}
\bla
\item Clearly $r_f(.,F):X\ra \mb{R}_{\geq 0}$ is a continuous and coercive function, where $r_f(x,F)=f((\|x-a\|)_{a\in F})$, when $F\in \mc{B}(X)$ and $f:\ell_\iy(F)\ra\mb{R}_{\geq0}$ is a continuous, monotone, and coercive function.  
\item In $(a)$, if we choose $\rho:F\ra\mb{R}_{\geq 0}$ a bounded weight then $r_\rho(.,F):(X,w)\ra\mb{R}_{\geq 0}$ is lower semicontinuous. Moreover $r_\rho (.,F):(X,w^*)\ra\mb{R}_{\geq 0}$ is also lower semicontinuous, when $X$ is a dual space.
\item For $F\in\mc{B}(X)$ and for any $\vartheta\in \ell_1(F)$, since $\vartheta$ induces a continuous, monotone functional on $\ell_\iy (F)$, it is easy to observe that $r_\vartheta (.,F):(X,w)\ra \mb{R}_{\geq 0}$ is lower semicontinuous and so is when $X$ is a dual space endowed with the $w^*$-topology.
\el
\end{Rem}

\begin{Pro}\label{P1}
Suppose that $X$ is a dual space, $Y$ is a $w^*$-closed subspace of $X$, $F\in\mc{B}(X)$ and $f:\ell_\iy(F)\ra\mb{R}_{\geq0}$ be a monotone, coercive function such that $r_f(.,F):X\ra\mb{R}_{\geq0}$ is $w^*$-lower semicontinuous. Then, $\emph{Cent}_Y^f(F)\neq\es$.
\end{Pro}

\begin{proof}
Let $(y_n)\ci Y$ be such that $r_f(y_n,F)\ra \tr{rad}_Y^f(F)$. As $f$ is coercive, the sequence $(y_n)$ is bounded. Hence there exists $y_0\in Y$ such that $y_\alpha\ra y_0$ in $w^*$-topology, for some subnet $(y_\al)$ of $(y_n)$. It is clear that $r_f(y_0,F)=\tr{rad}_Y^f(F)$ and hence $y_0\in \tr{Cent}_Y^f(F)$.
\end{proof}
The following is a variation of \cite[Proposition IV.1.2]{H}. For a $B\in \mc{B}(X)$ and $f:\ell_\iy (B)\ra\mb{R}_{\geq 0}$, $f$ is said to be {\it weakly strictly monotone} if $f(\varphi_1)<f(\varphi_2)$, whenever $\varphi_1(t)<\varphi_2(t)$, for all $t\in B$.

\begin{thm}\label{T1}
Let $X$ be an $L$-embedded space and $Y$ be a subspace of $X$, which is also $L$-embedded and $F\in\mc{B}(X)$. Then, 
\bla
\item for all monotone, coercive functions  $f:\ell_\iy(F)\ra\mb{R}_{\geq0}$, such that $r_f(.,F):X^{**}\ra\mb{R}_{\geq0}$ is $w^*$-lower semicontinuous, $\emph{Cent}_Y^f(F)\neq\es$.
\item for all weakly strictly monotone, coercive functions $f:\ell_\iy(F)\ra\mb{R}_{\geq0}$, such that $r_f(.,F):X^{**}\ra\mb{R}_{\geq0}$ is $w^*$-lower semicontinuous, $\emph{\textrm{Cent}}_Y^f(F)$ is weakly compact.
\el
\end{thm}
\begin{proof}
$(a)$. Let $f:\ell_\iy(F)\ra\mathbb{R}_{\geq0}$ be a monotone, coercive function such that $r_f(.,F)$ is $w^*$-lower semicontinuous. Let $P:X^{**}\ra X$ be the $L$-projection. Thus, from \cite[Theorem IV.1.2]{H}, we have $P(\overline{Y}^{w^*})\ci Y$. By Proposition~\ref{P1} there exists $y^{**}\in \ov{Y}^{w^*}$ such that $r_f(y^{**},F)=\textrm{rad}_{\ov{Y}^{w^*}}^f(F)$. Based on the monotonicity of $f$, we obtain,
\begin{equation}\label{E1}
\textrm{rad}_{Y}^f(F)\geq\tr{rad}_{\ov{Y}^{w^*}}^f(F)
=r_f(y^{**},F)
\geq r_f(Py^{**},F)
\geq\tr{rad}_Y^f(F).
\end{equation}
Since $Py^{**}\in Y$, our first assertion follows. 

$(b)$. Let $f:\ell_\iy(F)\ra\mathbb{R}_{\geq0}$ be a weakly strictly monotone, coercive function such that $r_f(.,F)$ is $w^*$-lower semicontinuous. By (\ref{E1}), we obtain a $y^{**}\in\ov{Y}^{w^*}$ such that, 
$r_f(Py^{**},F)=r_f(y^{**},F).$
Now for all $x\in F$,
\[
\|y^{**}-x\|=\|Py^{**}-x+(Id-P)y^{**}\|=\|Py^{**}-x\|+\|y^{**}-Py^{**}\|.
\]
This leads to $r_f(Py^{**},F)=f\left((\|Py^{**}-x\|+\|y^{**}-Py^{**}\|)_{x\in F}\right)$.

Since $f$ is weakly strictly monotone, we have $y^{**}=Py^{**}$ and consequently, $\textrm{Cent}_Y^f(F)=\textrm{Cent}_{\ov{Y}^{w^*}}^f(F)$. Hence, the $w^*$-closure of $\textrm{Cent}_Y^f(F)$ is contained in $X$, so it is weakly compact.
\end{proof}
As a corollary of the Theorem~\ref{T1}, we have,

\begin{Cor}\label{10}
\bla
\item Let $X$ be an $L$-embedded space and $Y$ be a subspace of $X$, which is also $L$-embedded. Suppose that $Y$ has $n.X.I.P.$ in $X$ for all $n\in\mb{N}$. Then, $Y$ is a central subspace of $X$.
\item Let $L_1(\Si,X)$ be an $L$-embedded space. Let $F\in \mc{B}(L_1(\Si,X))$ and $f:\ell_\iy(F)\ra\mb{R}_{\geq0}$ be a monotone, coercive function such that $r_f(.,F):L_1(\Si,X)^{**}\ra\mb{R}_{\geq0}$ is $w^*$-lower semicontinuous. Then, for all sub $\si$-algebra $\Si'\ci\Si$, $\emph{Cent}_{L_1(\Si',X)}^f(F)\neq\es$. 
\el
\end{Cor}

\begin{proof}
$(a).$ From our assumption, it follows that both $X$ and $Y$ admit weighted Chebyshev center for finite subsets of $Y$. The result now follows from Theorem~\ref{T9}, and Theorem~\ref{T1}.

$(b).$ Since $L_1(\Si',X)$ is one complemented in $L_1(\Si,X)$, it follows from \cite[Proposition~IV.1.5]{H}, that $L_1(\Si',X)$ is also $L$-embedded. Thus, our conclusion follows from Theorem \ref{T1}.
\end{proof}

Similar arguments stated in the proof of Theorem~\ref{T1}(b) also lead to the following.

\begin{lemma}\label{L2}
Let $X$ be an $L$-embedded space and $Y$ be a subspace of $X$, which is also $L$-embedded. Let $F\in\mc{B}(X)$ and $f:\ell_\iy(F)\ra\mb{R}_{\geq 0}$ be a strictly monotone, coercive function such that $r_f(.,F):X^{**}\ra\mb{R}_{\geq0}$ is $w^*$-lower semicontinuous,. Suppose that $(y_n)$ is a sequence in $Y$ such that $\lim_n r_f(y_n,F)=\emph{rad}_Y^f(F)$. Then, $(y_n)$ is relatively weakly compact. 
\end{lemma}

As a consequence of Lemma~\ref{L2}, we have the following.

\begin{thm}\label{T4}
Let $X$ be a Banach space such that $L_1(\Si,X)$ is $L$-embedded. Let $\{\Si_n\}_{i=1}^\iy$ be an increasing sequence of sub $\si$-algebras. Let $\Si_\iy$ be the $\si$-algebra generated by $\cup_{n\in\mathbb{N}} \Si_n$. Let $F\in \mc{B}(L_1(\Si,X))$ and $f:\ell_\iy(F)\ra\mb{R}_{\geq 0}$ be a  strictly monotone, coercive function such that $r_f(.,F):L_1(\Si,X)^{**}\ra\mb{R}_{\geq0}$ is $w^*$-lower semicontinuous. Suppose that $\emph{rad}_{L_1(\Si_n,X)}^f(F)=r_f(f_n,F)$ for some $f_n\in L_1(\Si_n,X)$. Then, $(f_n)$ is relatively weakly compact, and any weak limit point of this sequence belongs to $\emph{Cent}_{L_1(\Si_\iy, X)}^f(F)$.
\end{thm}
\begin{proof}
We follow the techniques similar to those used in \cite[Theorem~12]{T}. Let $E_n$ be the conditional expectation projection of $L_1(\Si_\iy,X)$ onto $L_1(\Si_n,X)$ for all $n\in\mathbb{N}$. It is known that $E_n(g)\ra g$ for all $g\in L_1(\Si_\iy,X)$. We have for any $n\in\mb{N}$,
$\textrm{rad}_{L_1(\Si_\iy,X)}^f(F)\leq\textrm{rad}_{L_1(\Si_n,X)}^f(F)\leq r_f(f_n,F)$
and hence $\textrm{rad}_{L_1(\Si_\iy,X)}^f(F)\leq\liminf_n r_f(f_n,F)$.

Now for any $g\in L_1(\Si_\iy,X)$, we have,
\beqa
\limsup_n~ \textrm{rad}_{L_1(\Si_n,X)}^f(F)\leq \lim_n r_f(E_n(g),F)\leq r_f(g,F).
\eeqa

Hence, $\limsup_n r_f(f_n,F)=\limsup_n \textrm{rad}_{L_1(\Si_n,X)}^f(F)\leq\textrm{rad}_{L_1(\Si_\iy,X)}^f(F)$.

Thus, $\textrm{rad}_{L_1(\Si_\iy,X)}^f(F)=\lim_n r_f(f_n,F)=\lim_n \textrm{rad}_{L_1(\Si_n,X)}^f(F)$. Hence, by Lemma \ref{L2}, $(f_n)$ is relatively weakly compact. Using the $w^*$-lower semi-continuity of $r_f(.,F)$, we have the weak limit point of $(f_n)$ belonging to $\textrm{Cent}_{L_1(\Si_\iy, X)}^f(F)$.
\end{proof}

Our next observation follows from \cite[Theorem~5.1]{LV} and the existence of the conditional expectation operator.

\begin{thm}\label{T2}
Let $X$ denote a Banach space such that $X^*$ has $RNP$. Let $(\Om,\Si,\mu)$ be a finite measure space. Let $F\in \mc{B}(L_1(\Si,X^*))$ and $f:\ell_\iy(F)\ra\mb{R}_{\geq0}$ be a monotone, coercive function such that $r_f(.,F): L_1(\Si,X^*)^{**}\ra\mb{R}_{\geq0}$ is $w^*$-lower semicontinuous. Then, for all sub $\si$-algebra $\Si'\ci\Si$, $\emph{Cent}_{L_1(\Si',X^*)}(F)\neq\es$.
\end{thm}

Let us recall that $\mc{WC}(X)$ denotes the set of all weakly compact subsets of $X$. Additionally, let us also recall the following theorem of Dunford (see \cite[pg. 101]{D}).

\begin{thm}\label{T14}
Let $X$ be a reflexive Banach space. Then, $K\ci L_1(\mu,X)$ is relatively weakly compact if and only if K is bounded and uniformly integrable.
\end{thm}
It is apparent from the proof of the above Theorem that for any Banach space $X$, a relatively weakly compact subset of $L_1(\mu,X)$ is bounded and uniformly integrable.

\begin{thm}\label{T11}
	Let $Y$ be a reflexive subspace of $X$ and $(\Om,\Si,\mu)$ be a finite measure space. 
	\bla
	\item  Then for $F\in \mc{WC}(L_1(\Si,X))$, for all sub $\si$-algebra $\Si'\ci\Si$ and for all monotone, coercive function $f:\ell_\iy(F)\ra\mb{R}_{\geq0}$ such that $r_f(.,F):L_1(\Si,X)\ra\mb{R}_{\geq0}$ is $w$-lower semicontinuous, $\emph{Cent}_{L_1(\Si',Y)}^f(F)\neq\es$.
	\item  Then for all $F\in \mc{B}(L_\iy(\Si,X))$, monotone, coercive function $f:\ell_\iy(F)\ra\mb{R}_{\geq0}$ such that $r_f(.,F):L_1(\Si,X^*)^*\ra\mb{R}_{\geq0}$ is $w^*$-lower semicontinuous,
	$\emph{Cent}_{L_\iy(\Si,Y)}^f(F)\neq\es$.
	\el
\end{thm}
\begin{proof}
	$(a).$  Let  $(g_n)\ci L_1(\Si',Y)$ be a  such that $r_f(g_n, F)\ra\textrm{rad}^f_{L_1(\Si',Y)}(F)$. Clearly, $(g_n)$ is bounded. Hence by a vector-valued version of  Kadec-Pe\l czyn\`ski-Rosenthal theorem (\cite[Lemma~2.1.3]{C}), there exists a subsequence
	$(g_{n_k})$ of $(g_n)$ and a sequence of pairwise disjoint sets $(A_k)\ci\Si'$ such that $(g_{n_k}\chi_{A^c_k})$ is uniformly integrable in $L_1(\Si',Y)$. Since, $\sum_{k=1}^\iy\mu(A_k)=\mu(\cup_{k=1}^\iy A_k)\leq\mu(\Om)<\iy$, we have $\mu(A_k)\ra 0$. 
	
	Let us choose $h\in F$ then,
	\[
	\|g_{n_k}\chi_{A_k^c}-h\|_1=\int_{A_k}\|h\|d\mu+\int_{A_k^c}\|g_{n_k}-h\|d\mu
	\leq\int_{A_k}\|h\|d\mu+\|g_{n_k}-h\|_1.
	\]
	As $F$ is uniformly integrable, we get $\lim_k \int_{A_k}\|h\|d\mu=0$ uniformly for all $h\in F$. Now, it is easy to see that  $r_f(g_{n_k}\chi_{A_k^c},F)\ra\textrm{rad}^f_{L_1(\Si',Y)}(F)$. Also since $(g_{n_k}\chi_{A_k^c})$ is a bounded sequence in $L_1(\Si',Y)$, by Theorem~\ref{T14}, $(g_{n_k}\chi_{A_k^c})$ is relatively weakly compact in $L_1(\Si',Y)$. Let us denote by $(g_{n_{k(j)}})$, a weakly convergent subsequence of $(g_{n_{k}}\chi_{A_k^c})$ converging weakly to $g\in L_1(\Si',Y)$. Then $r_f(g,F)\leq\lim\inf_j r_f(g_{n_{k(j)}},F)=\tr{rad}_{L_1(\Si',Y)}^f(F)$. Hence $r_f(g,F)=\textrm{rad}^f_{L_1(\Si',Y)}(F)$.
	
	$(b).$ We consider $L_\iy(\Si,X)$ to be canonically embedded in $L_1(\Si,X^*)^*$. It follows from from Proposition $3.1$ of \cite{TP1} that $L_\iy(\Si,Y)$ is $w^*$-closed in $L_1(\Si,X^*)^*$. Hence by the $w^*$-lower semicontinuity of $r_f(.,F)$, we have an $g\in L_\iy(\Si,Y)$ such that $r_f(g,F)=\textrm{rad}_{L_\iy(\Si,Y)}^f(F)$.
\end{proof}

\begin{Rem}
Using similar arguments, it can be shown that for a reflexive subspace $Y$ of $X$ and for a sub-$\si$-algebra $\Si'\ci\Si$, $(L_1(\Si,X), B_{L_1(\Si',Y)},\mc{WC}(L_1(\Si,X)))$ has \emph{\bf {rcp}}.
\end{Rem}

Let $\mr{F}$ be the family of all monotone functions $f:(\mb{R}^n_{\geq 0},\|.\|_\iy)\ra \mb{R}_{\geq 0}$, where $n\geq 1$.

\begin{Pro}\label{P2}
Let $P$ be a projection on $X$ of norm-$1$ and $Y=\ker (P)$. Let $Z\ci P(X)$ be a subspace. If $(X,Y+Z,\mc{F}(X))$ has \emph{\bf $\mr{F}$-rcp}, then $(P(X),Z,\mc{F}(P(X)))$ has \emph{\bf $\mr{F}$-rcp}.
\end{Pro}
\begin{proof}
Let $F\in \mc{F}(P(X))$ and $card(F)=n$. Suppose that, $f\in\mr{F}$, where $f:\mb{R}^n_{\geq 0}\ra\mb{R}_{\geq 0}$ be monotone. Based on our assumption, there exists $y_0+z_0\in Y+Z$ such that $r_f(y_0+z_0,F)=\textrm{rad}_{Y+Z}^f(F)\leq\textrm{rad}_Z^f(F)$. 
Then, $r_f(z_0,F)=r_f(P(y_0+z_0),F)
\leq r_f(y_0+z_0,F)
\leq\textrm{rad}_Z^f(F)$.
\end{proof}

\begin{Rem}
\bla
\item An identical arguments stated in the proof of Proposition~\ref{P2} also work to conclude that  $F$ has $\mr{F}$-{\bf rcp} in $Z$ whenever $F\in \mc{B}(P(X))$ and $F$ has $\mr{F}$-{\bf rcp} in $Y+Z$. Here we consider $\mr{F}$ to be the family $\ell_\iy(F)$.
\item Similar conditions as that in Proposition~\ref{P2}, it can be used to prove that if $(X,B_{Y+Z},\mc{B}(X))$ has {\bf rcp}, then $(P(X),B_Z,\mc{B}(P(X)))$ has {\bf rcp}.
\el
\end{Rem}

\end{document}